\documentclass[a4paper,12pt,reqno]{amsart}

\usepackage{amsmath,amsfonts,amsthm,amssymb,color}
\usepackage[latin1]{inputenc}
\usepackage{xy} 
\usepackage{pdfsync}
\usepackage{hyperref}

\newcommand{\cacha}{\Hat{\mathcal{C}}}

\newcommand{\norm}[1]{\lVert #1\rVert}

\newcommand{\ka}{\kappa}

\newcommand{\iot}{\int_{0}^{t}}

\newcommand{\ott}{[0,T]}

\DeclareMathOperator{\id}{\text{Id}}

\DeclareMathOperator{\Id}{Id}

\makeatletter
\newcommand{\eqcolon}{\mathrel{\mathord{=}\raise.2\p@\hbox{:}}}
\newcommand{\coloneq}{\mathrel{\raise.2\p@\hbox{:}\mathord{=}}}
\makeatother

\newcommand{\D}{\mathbb D}
\newcommand{\R}{\mathbb R}
\newcommand{\N}{\mathbb N}

\newcommand{\be}{\mathbf{E}}

\newcommand{\bp}{\mathbf{P}}

\newcommand{\cb}{\mathcal B}
\newcommand{\cac}{\mathcal C}

\newcommand{\cd}{\mathcal D}
\newcommand{\ce}{\mathcal E}
\newcommand{\cf}{\mathcal F}

\newcommand{\ch}{\mathcal H}

\newcommand{\ck}{\mathcal K}
\newcommand{\cl}{\mathcal L}
\newcommand{\cn}{\mathcal N}

\newcommand{\cs}{\mathcal S}

\newcommand{\crr}{\mathcal R}

\newcommand{\al}{\alpha}
\newcommand{\ep}{\varepsilon}

\newcommand{\ga}{\gamma}
\newcommand{\la}{\lambda}

\newcommand{\om}{\omega}
\newcommand{\oom}{\Omega}

\newcommand{\vp}{\varphi}

\newcommand{\lp}{\left(}
\newcommand{\rp}{\right)}
\newcommand{\lc}{\left[}
\newcommand{\rc}{\right]}
\newcommand{\lcl}{\left\{}
\newcommand{\rcl}{\right\}}
\newcommand{\lln}{\left|}
\newcommand{\rrn}{\right|}
\newcommand{\lla}{\left\langle}
\newcommand{\rra}{\right\rangle}

\newcommand{\bean}{\begin{eqnarray*}}
\newcommand{\eean}{\end{eqnarray*}}
\newcommand{\ben}{\begin{enumerate}}
\newcommand{\een}{\end{enumerate}}
\newcommand{\beq}{\begin{equation}}
\newcommand{\eeq}{\end{equation}}

\newtheorem{theorem}{Theorem}[section]

\newtheorem{corollary}[theorem]{Corollary}

\newtheorem{definition}[theorem]{Definition}

\newtheorem{hypothesis}{Hypothesis}
\newtheorem{lemma}[theorem]{Lemma}

\newtheorem{proposition}[theorem]{Proposition}

\theoremstyle{remark}
\newtheorem{remark}[theorem]{Remark}

\begin{document}

\title{Malliavin calculus for fractional heat equation}
\author{A. Deya \and S. Tindel}
\date{\today}
\begin{abstract}
In this article, we give some existence and smoothness results for the law of the solution to a stochastic heat equation driven by a finite dimensional fractional Brownian motion with Hurst parameter $H>1/2$. Our results rely on recent tools of Young integration for convolutional integrals combined with stochastic analysis methods for the study of laws of random variables defined on a Wiener space.
\end{abstract}

\address{Aur{\'e}lien Deya, Institut {\'E}lie Cartan Nancy, Universit\'e de Nancy 1, B.P. 239,
54506 Vand{\oe}uvre-l{\`e}s-Nancy Cedex, France}
\email{deya@iecn.u-nancy.fr}

\address{Samy Tindel, Institut {\'E}lie Cartan Nancy, Universit\'e de Nancy 1, B.P. 239,
54506 Vand{\oe}uvre-l{\`e}s-Nancy Cedex, France}
\email{tindel@iecn.u-nancy.fr}

\thanks{S. Tindel is partially supported by the (French) ANR grant ECRU}

\date{\today}

\maketitle

\begin{center}
 \emph{Dedicated to David Nualart on occasion of his 60th birthday}
\end{center}

\section{Introduction}

The definition and resolution of evolution type PDEs driven by general H\"older continuous signals have experienced tremendous progresses during the last past years. When the H\"older regularity of the driving noise is larger than $1/2$, this has been achieved thanks to Young integrals \cite{GLT} or fractional integration \cite{Maslo-Nualart} techniques. The more delicate issue of a H\"older exponent smaller than $1/2$ has to be handled thanks to rough paths techniques, either by smart transformations allowing to use limiting arguments \cite{car-friz,car-friz-ober,friz-ober-1} or by an adaptation of the rough paths formalism to evolution equations \cite{RHE-glo, RHE, GT}. Altogether, those contributions yield a reasonable definition of rough parabolic PDEs, driven at least by a finite dimensional signal.

\smallskip

With those first results in hand, a natural concern is to get a better understanding of the processes obtained as solutions to stochastic PDEs driven by rough signals. This important program includes convergence of numerical schemes (see \cite{NSRHE} for a result in this direction), ergodic properties and a thorough study of the law of those processes. The current article makes a first step towards the last of these items.

\smallskip

Indeed, we shall consider here a simple case of rough evolution PDE and see what kind of result might be obtained as far as densities of the solution are concerned. More specifically, we focus on the following mild heat equation on $(0,1)$
\begin{equation}\label{eq-gene-fbm}
Y_t=S_t \vp +\int_0^t S_{t-u}(F_i(Y)_u) \, dB^i_u \, , \quad   \ t\in [0,T],
\end{equation}
where $T>0$ is a finite horizon, $S_t$ stands for the heat semi-group associated with Dirichlet boundary conditions, $\vp$ is a smooth enough initial condition, $F_i: L^2(0,1) \to L^2(0,1)$ and $B: [0,T] \to \R^d$ is a $d$-dimensional fractional Brownian motion with Hurst parameter $H >1/2$. For this equation, we obtain the following results:
\begin{enumerate}
\item 
Existence of a density for the random variable $Y_t(\xi)$ for any $t\in(0,T]$ and $\xi\in(0,1)$, when the $F_i$'s are rather general Nemytskii operators $F_i(\vp)(\xi):=f_i(\vp(\xi))$. See Theorem \ref{theo-exi} for a precise statement.
\item
When the $F_i$'s are defined through some regularizing kernel (see Hypothesis \ref{hyp:regularized-young-eq}), we obtain that the density of $Y_t(\xi)$ is smooth. This will be the content of Theorem~\ref{thm:regu-density}.
\end{enumerate}
To the best of our knowledge, these are the first density results for solutions to nonlinear PDEs driven by fractional Brownian motion. Let us point out that we could have obtained the same kind of results for a more general class of equations (operator under divergence form, general domain $D\subset\R^n$, drift term, Gaussian process as driving noise). We prefer however to stick to the simple case of the fBm-driven stochastic heat equation for the sake of readability and conciseness.

\smallskip

Our main results will obviously be based on a combination of pathwise estimates for integrals driven by rough signals and Malliavin calculus tools. In particular, a major part of our effort will be dedicated to the differentiation of the solution to equation~\eqref{eq-gene-fbm} with respect to the driving noise $B$ and to a proper estimate of the derivative. Since the equations for derivatives are always of linear type they lead to exponential type estimates, which are always a delicate issue. This is where we shall consider some regularizing vector fields $F_i$ in  \eqref{eq-gene-fbm}, and proceed to a careful estimation procedure (see Section \ref{sec:estim-solution}). It should also be noticed at this point that the basis of our stochastic analysis tools is contained in the celebrated book \cite{Nua} by D. Nualart, plus the classical reference \cite{nua-sau} as far as equations driven by fBm are concerned.

\smallskip

Here is how our article is structured: Section \ref{sec:setting} is devoted to recall basic facts on both pathwise noisy evolution equations and Malliavin calculus for fractional Brownian motion. We differentiate the solution to \eqref{eq-gene-fbm} and obtain the existence of the density at Section \ref{sec:existence-density}. Finally, further estimates on the Malliavin derivative and smoothness of the density are derived at Section \ref{sec:smoothness-density}.

\

\noindent
\textbf{Notation}. Throughout the paper, we will use the generic notation $c$ to refer to the constants that only depend on non-significant parameters. The constants which are to play a more specific role in our reasoning will be labelled $c_1,c_2,$...

\smallskip

For any $k\in \N$, we will denote by $\cac^{k,\textbf{b}}(\R;\R)$ the space of functions on $\R$ which are $k$-times differentiable with bounded derivatives. For any $\ga\in (0,1)$, $\cac^\ga=\cac^\ga([0,T];\R^d)$ will stand for the set of ($d$-dimensional) $\ga$-Hölder paths on $[0,T]$.

\section{Setting}
\label{sec:setting}

One of the technical advantages of dealing with the simple case of a stochastic heat equation on $(0,1)$ is a simplification in the functional analysis setting, based on rather elementary Fourier series considerations (notice in particular that the $L^p$ considerations of \cite{RHE} can be avoided). We shall first detail this setting, and then recall some basic facts on equations driven by noisy signals and fractional Brownian motion. Throughout the section, we assume that a (finite) horizon $T$ has been fixed for the equation.

\subsection{Fractional Sobolev spaces}
As mentioned above, we are working here with the heat equation in the Hilbert space $\cb:=L^2(0,1)$ with Dirichlet boundary conditions. The Laplace operator $\Delta$ on $\cb$ can be diagonalized in the orthonormal basis 
$$ e_n(\xi):=\sqrt{2} \sin(\pi n \xi)\ \ (n\in \N^\ast), \quad \text{with eigenvalues} \ \  \la_n:= \pi^2n^2.$$
We shall denote by $(y^n)_n$ the (Fourier) decomposition of any function $y\in\cb$ on this orthonormal basis.

\smallskip

Sobolev spaces based on $\cb$ are then easily characterized by means of Fourier coefficients. We label their definition for further use:

\begin{definition}\label{defi:sobo-frac-tore}
For any $\al \geq 0$, we denote by $\cb_\al$ the fractional Sobolev space of order $\al$ based on $\cb$, defined by
\begin{equation}
\cb_\al :=\lcl y \in L^2(0,1): \ \sum_{n=1}^\infty \la_n^{2\al} (y^n)^2 < \infty \rcl.
\end{equation}
This space is equipped with its natural norm
$\norm{y}_{\cb_\al}^2:=\norm{\Delta^\al y}_\cb^2=\sum_{n=1}^\infty \la_n^{2\al}(y^n)^2$.
We also set $\cb_{\infty}=\cac(0,1)$.
\end{definition}

The above-defined fractional Sobolev spaces enjoy the following classical properties (see \cite{Adams,run-sick}):

\begin{proposition}
Let $\cb_\al,\cb_{\infty}$ be the Sobolev spaces introduced at Definition \ref{defi:sobo-frac-tore}. Then the following hold true:

\smallskip

\noindent $\bullet$
\emph{Sobolev inclusions}: If $\al >1/4$, then we have the continuous embedding
\begin{equation}\label{sobol-incl}
\cb_{\al}\subset \cb_\infty.
\end{equation}

\smallskip

\noindent $\bullet$ \emph{Algebra}: If $\al  > 1/4$, then $\cb_{\al}$ is a Banach algebra with respect to pointwise multiplication, or in other words 
\begin{equation}\label{algebra}
\norm{\varphi \cdot \psi}_{\cb_{\al}} \leq \norm{\varphi}_{\cb_{\al}} \norm{\psi}_{\cb_{\al}}.
\end{equation}

\smallskip

\noindent $\bullet$ \emph{Composition}: If $0\leq \al <1/2$, $\vp\in \cb_\al$ and $f:\R \to \R$ belongs to $\cac^{1,\textbf{b}}$, then $f(\vp) \in \cb_\al$ and
\begin{equation}\label{compos}
\norm{f(\vp)}_{\cb_\al} \leq c_f \lcl 1+\norm{\vp}_{\cb_{\al}} \rcl.
\end{equation}
Here, $f(\vp)$ is naturally understood as $f(\vp)(\xi):=f(\vp(\xi))$.
\end{proposition}

\smallskip

Let now $S_t$ be the heat semigroup associated with $\Delta$, and notice that if an element $y\in L^2(0,1)$ can be decomposed as $y=\sum_{n\ge 1}y^n e_n$, then $S_t y=\sum_{n\ge 1} e^{-\la_n t}y^n e_n$.
The general theory of fractional powers of operators provides us with sharp estimates for the semigroup $S_t$ (see for instance \cite{pazy}):
\begin{proposition}\label{prop-sem-sob}
The heat semigroup $S_t$ satisfies the following properties:

\smallskip

\noindent $\bullet$ \emph{Contraction}: For all $t \geq 0$, $\al \geq 0$, $S_t$ is a contraction operator on $\cb_{\al}$.

\smallskip

\noindent $\bullet$ \emph{Regularization}: For all $t \in (0,T]$, $\al\geq 0$, $S_t$ sends $\cb$ on $\cb_{\al}$ and 
\begin{equation} \label{regu-prop-semi}
\norm{S_t \varphi}_{\cb_{\al}} \leq c_{\al,T}\, t^{-\al} \norm{\varphi}_{\cb}.
\end{equation}

\smallskip

\noindent $\bullet$ \emph{Hölder regularity}. For all $t\in (0,T]$, $\varphi \in \cb_{\al}$,
\begin{equation}\label{regu-hold-semi}
\norm{S_t\varphi-\varphi}_{\cb} \leq c_{\al,T}\, t^\al \norm{\varphi}_{\cb_{\al}}.
\end{equation}
\end{proposition}

\subsection{Young convolutional integrals}
The stochastic integrals involved in equation~\eqref{eq-gene-fbm} will all be understood in the Young sense. In order to define them properly, let us first introduce some notation concerning H\"older type spaces in time. To begin with, for any $\al\geq 0$ and any subinterval $I\subset [0,T]$, set $\cac^0(I;\cb_\al)$ for the space of continuous $\cb_\al$-valued functions on $I$, equipped with the supremum norm. Then H\"older spaces of $\cb_\al$-valued functions can be defined as follows: for $\ka\in (0,1)$, set
\begin{equation*}
\cac^\ka(I;\cb_{\al}):=\lcl y \in \cac^0(I;\cb_{\al}): \ \sup_{s<t \in I} \frac{\norm{y_t-y_s}_{\cb_{\al}}}{\lln t-s \rrn^\ka} < \infty \rcl.
\end{equation*} 
Observe now that the definition of our stochastic integrals weighted by the heat semigroup will require the introduction of a small variant of those H\"older spaces (see \cite{GT,RHE} for further details): we define $\cacha^\ka(I;\cb_{\al})$ as
\begin{equation*}
\cacha^\ka(I;\cb_{\al}):=\lcl y \in \cac^0(I;\cb_{\al}): \ \sup_{s<t\in I} \frac{\norm{y_t- S_{t-s} \, y_s}_{\cb_{\al}}}{\lln t-s \rrn^\ka} < \infty \rcl.
\end{equation*} 
In order to avoid confusion, the natural norms on the spaces $\cac^\ka(I;\cb_\al)$, $\cacha^\ka(I;\cb_\al)$ are respectively denoted by $\cn[\cdot;\cac^\ka(I;\cb_\al)]$, $\cn[\cdot;\cacha^\ka(I;\cb_\al)]$, etc. For the sake of conciseness, we shall often write $\cac^\ka(\cb_\al)$ (resp. $\cacha^\ka(\cb_\al)$) instead of $\cac^\ka([0,T];\cb_\al)$ (resp. $\cacha^\ka([0,T];\cb_\al)$). We also need to introduce a family of spaces $\cacha^{0,\ka}(I;\cb_\ka)$ in the following way:
\begin{lemma}\label{lem:imbed-cacha}
For any $\ka \in (0,1)$ and any subinterval $I\subset [0,T]$, let $\cacha^{0,\ka}(I;\cb_\ka)$ be the space associated with the norm
$$\cn[\cdot;\cacha^{0,\ka}(I;\cb_\ka)] :=\cn[\cdot;\cac^0(I;\cb_\ka)]+ \cn[\cdot;\cacha^\ka(I;\cb_\ka)].$$
Then the following continuous embedding holds true:
\begin{equation}\label{lien-der-delha}
\cacha^{0,\ka}(I;\cb_\ka) \subset \cac^\ka(I;\cb).
\end{equation}
More generally, for every $\la \geq \ka$, 
\begin{equation}\label{lien-der-delha-2}
\cn[y;\cac^\ka(I;\cb)] \leq \cn[y;\cacha^\ka(I;\cb_\la)]+c_\la \lln I\rrn^{\la-\ka} \cn[y;\cac^0(I;\cb_\la)].
\end{equation}
\end{lemma}

\begin{proof}
Indeed, owing to (\ref{regu-hold-semi}), one has, for every $s<t\in I$,
$$\norm{y_t-y_s}_\cb \leq \norm{y_t-S_{t-s} y_s}_\cb+\norm{(S_{t-s}-\Id) y_s}_\cb \leq \norm{y_t-S_{t-s} y_s}_{\cb_\la}+c_\la \lln t-s \rrn^\la \norm{y_s}_{\cb_\la}.$$

\end{proof}

With those definitions in hand, the following proposition (borrowed from \cite{RHE}) will be invoked in the sequel in order to give a meaning to our  stochastic integrals weighted by the heat semigroup:
\begin{proposition}\label{prop:yg-conv-int}
Consider a $\ga$-H\"older real-valued function $x$ defined on $[0,T]$. Let $I=[\ell_1,\ell_2]$ be a subinterval of $[0,T]$ and fix $\ka \in [0,\ga]$ such that $\ga+\ka>1$. Suppose that $z\in \cac^0(I;\cb_\la) \cap \cac^\ka(I;\cb_{\la-\al})$ for some parameters $\la \geq 0$, $0\leq \al \leq \min(\ka, \la)$. Then, for every $s<t \in I$, the convolutional Riemann sum
$$\sum_{t_k \in \Pi} S_{t-t_{k+1}} z_{t_{k+1}} \lc x_{t_{k+1}}-x_{t_k}\rc
$$
converges in $\cb_\la$ as the mesh of the partition $\Pi:=\{s=t_0 <t_1 < \ldots < t_n=t \}$ tends to $0$, and we denote the limit by $\int_s^t S_{t-u} z_u \, dx_u$.

\smallskip

\noindent
Moreover, for every $\vp \in \cb_\la$, there exists a unique path $y\in \cacha^\ga(I;\cb_\la)$ such that $y_{\ell_1}=\vp$ and $y_t-S_{t-s}y_s=\int_s^t S_{t-u}(z_u) \, dx_u$ if $s<t\in I$. For this function, the following estimate holds:
\begin{equation}\label{esti-yg}
\cn[y;\cacha^\ga(I;\cb_\la)] \leq c \norm{x}_\ga \lcl \cn[z;\cac^0(I;\cb_\la)]+\lln I\rrn^{\ka-\al} \cn[z;\cac^\ka(I;\cb_{\la-\al})] \rcl,
\end{equation}
for some constant $c$ that only depends on $(\ga,\ka,\la,\al)$.




\end{proposition}

\subsection{Malliavin calculus techniques}
This subsection is devoted to present the Malliavin calculus setting which we shall work in, having in mind the differentiability properties of the solution to \eqref{eq-gene-fbm}. 

\subsubsection{Wiener space associated to fBm}
Let us first be more specific about the probabilistic setting in which we will work.
For some fixed $H\in(1/2,1)$, we consider $(\oom,\cf,\bp)$ the canonical probability space associated with the fractional
Brownian motion with Hurst parameter $H$. That is,  $\oom=\cac_0([0,T];\R^{d})$ is the Banach space of continuous functions
vanishing at $0$ equipped with the supremum norm, $\cf$ is the Borel sigma-algebra and $\bp$ is the unique probability
measure on $\oom$ such that the canonical process $B=\{B_t, \; t\in [0,T]\}$ is a $d$-dimensional fractional Brownian motion with Hurst
parameter $H$, with covariance function 
\begin{equation}\label{eq:cov-fbm}
\be \big[  B_t^i \, B_s^j \big] = \frac 12 \lp t^{2H} + s^{2H} - |t-s|^{2H}  \rp \, \mathbf{1}_{(i=j)},
\qquad s,t\in [0,T].
\end{equation}
In particular, the paths of $B$ are almost surely $\ga$-H\"older continuous for all $\ga \in (0,H)$.

\smallskip

Consider then a fixed parameter  $H>1/2$, and let us start by briefly describing the abstract 
Wiener space introduced for Malliavin calculus purposes (for a more general and complete description, we refer the reader to \cite[Section 3]{nua-sau}). 

\smallskip

Let $(e_1,\ldots,e_d)$ be the canonical basis of $\R^d$, $\ce$ be the set of $\R^{d}$-valued step functions on $[0,T]$ and $\ch$ the completion of $\ce$ with respect to the semi-inner product
\[
\langle \mathbf{1}_{[0,t]} \, e_{i}, \mathbf{1}_{[0,s]} \, e_{j}\rangle_\ch := R_H(s,t) \, \mathbf{1}_{(i=j)}, \qquad s,t \in [0,T].
\]
Then, one constructs an isometry $K^*_H: \ch \rightarrow  L^2([0,T];\R^d)$  such that $K^*_H(\mathbf{1}_{[0,t]}\, e_{i}) = \mathbf{1}_{[0,t]}$  $K_H(t,\cdot)\,e_{i}$, 
where the kernel $K=K_H$ is given by 
\[
K(t,s)= c_H s^{\frac 12 -H} \int_s^t (u-s)^{H-\frac 32} u^{H-\frac 12} \, du
\]
and verifies that $\be[B_s^{i}\, B_t^{i}]= \int_0^{s\land t} K(t,r) K(s,r)\, dr$, for some constant $c_H$. Moreover, let us observe that $K^*_H$ can be represented 
in the following form:
\begin{equation}\label{eq:def-K-star}
[K_H^* \vp]_t = \int_t^T \vp_r \partial_r K(r,t) \, dr
=d_H t^{H-1/2} \lc  I_{T^-}^{H-1/2} \lp u^{H-1/2} \vp \rp \rc_t,
\end{equation}
where $I_{T^-}^{\al}$ stands for the fractional integral of order $\al$.
The fractional Cameron-Martin space can be introduced in the following way: let 
$\ck_H : L^2([0,T];\R^d) \rightarrow \ch_H := \ck_H(L^2([0,T];\R^d))$ be the operator defined by
\[
[\ck_H h](t) := \int_0^t K(t,s) \, h(s)\, ds, \qquad h\in L^2([0,T];\R^d).
\]
Then, $\ch_H$ is the Reproducing Kernel Hilbert space associated with the fractional Brownian motion $B$. Observe that, in the case of the 
classical Brownian motion, one has that $K(t,s)=\mathbf{1}_{[0,t]}(s)$, $K^*$ is the identity operator in $L^2([0,T];\R^d)$ and $\ch_H$ is the usual Cameron-Martin 
space. 

\smallskip

In order to deduce that $(\oom,\ch,\bp)$ defines an abstract Wiener space, we remark that $\ch$ is continuously and densely embedded in $\oom$. In fact, 
one proves that the operator $\crr_H :\ch \rightarrow \ch_H$ given by
\[
\crr_H \psi := \int_0^\cdot K(\cdot,s) [K^* \psi](s)\, ds
\]
defines a dense and continuous embedding from $\ch$ into $\oom$; this is due to the fact that $\crr_H \psi$ is $H$-H\"older continuous (for details, see 
\cite[p. 400]{nua-sau}).  

\smallskip

Let us also recall that there exists a $d$-dimensional Wiener process $W$ defined on $(\oom,\ch,\bp)$ such that $B$ can be expressed as 
\begin{equation}\label{eq:volterra-representation}
B_t=\iot K(t,r) \, dW_r, \quad t\in\ott.
\end{equation}
This formula will be referred to as Volterra's representation of fBm.

\subsubsection{Malliavin calculus for $B$}
Let us introduce now the Malliavin derivative operator on the Wiener space $(\oom,\ch,\bp)$. Namely, we first let $\cs$ be the 
family of smooth functionals $F$ of the form
\[
F=f(B(h_1),\dots,B(h_n)),
\]
where $h_1,\dots,h_n\in \ch$, $n\geq 1$, and $f$ is a smooth function having polynomial growth together with all its  partial derivatives. 
Then, the Malliavin derivative of such a functional $F$ is the $\ch$-valued random variable defined by
\[
\cd F= \sum_{i=1}^n \frac{\partial f}{\partial x_i} (B(h_1),\dots,B(h_n)) h_i.
\]
For all $p>1$, it is known that the operator $\cd$ is closable from $L^p(\oom)$ into $L^p(\oom; \ch)$ (see e.g. \cite[Chapter 1]{Nua}).
We will still denote by $\cd$ the closure of this operator, whose domain is usually denoted by $\D^{1,p}$ and is defined 
as the completion of $\cs$ with respect to the norm
\[
\|F\|_{1,p}:= \left( E(|F|^p) + E( \|\cd F\|_\ch^p ) \right)^{\frac 1p}.
\]
Sobolev spaces $\D^{k,p}$ for any $k\in\N$ and $p\ge 1$ can be defined in the same way, and we denote by $\D^{k,p}_{{\rm loc}}$ the set of random variables $F$ for which there exists a sequence $(\oom_n,F_n)_{n\geq 1} \subset \cf \times \D^{k,p}$ such that $\oom_n \uparrow \oom$ a.s. and $F=F_n$ a.s. on $\oom_n$. We also set $\D^{\infty}=\cap_{k,p}\D^{k,p}$.

\begin{remark}\label{rmk:partial-mall-deriv}
For $F\in\D^{1,2}$, one can write $\cd F=\sum_{j=1}^{d}\cd^{j}F \, e_j$, where $\cd^{j}F$ denotes the Malliavin derivative with respect to the $j\textsuperscript{th}$ component of $B$.
\end{remark}

\smallskip

Since we deal with pathwise equations, we shall also be able to differentiate them in a pathwise manner. The relation between almost sure and Malliavin derivatives has been established by Kusuoka, and we quote it according to \cite[Proposition 4.1.3]{Nua}.
\begin{proposition}
A random variable $F$ is said to be $\ch$-differentiable if for almost all $\om\in \oom$ and for any $h\in \ch$, the map $\nu \mapsto F(\om + \nu \crr_H h)$ is differentiable. Those random variables belong to the space $\D^{1,p}_{{\rm loc}}$, for any $p>1$. Moreover, the following relation holds true:
\begin{equation}\label{eq:rel-malliavin-pathwise-deriv}
\langle \cd F,h\rangle_\ch =  D F(B) (\crr_H h), \qquad h\in \ch,
\end{equation}
where we recall that $\cd$ stands for the Malliavin derivative and $D$ for the pathwise differentiation operator.
\end{proposition}

\smallskip

Stochastic analysis techniques are widely used in order study laws of random variables defined on a Wiener space. Let us recall the main criterions we shall use in this direction:
\begin{proposition}\label{prop:density-criterions}
Let $F$ be a real-valued random variable defined on $(\oom,\cf,\bp)$. Then

\smallskip

\noindent\emph{(i)} If $F\in\D^{1,p}_{{\rm loc}}$ for $p>1$ and $\|\cd F\|_{\ch}>0$ almost surely, then the law of $F$ admits a density $p$ with respect to Lebesgue measure.

\smallskip

\noindent\emph{(ii)} If  $F\in\D^{\infty}$ and $\be[\|\cd F\|_{\ch}^{-p}]$ is finite for all $p\ge 1$, then the density $p$ of $F$ is infinitely differentiable.
\end{proposition}

\section{Existence of the density in the case of Nemytskii-type vector fields}\label{sec:existence-density}

In this section, we first consider a general equation of the form
\begin{equation}\label{eq-gene}
y_t=S_t \vp +\int_0^t S_{t-u}(F_i(y)_u) \, dx^i_u \quad , \quad \vp \in L^2(0,1) \ , \ t\in [0,T],
\end{equation}
driven by a $d$-dimensional noise $x=(x^1,\ldots,x^d)$ considered as a $\cac^{\ga}$ function with $\ga\in(1/2,1)$. We shall be able to handle the general case of a perturbation involving Nemytskii operators, that is
$$F_i(\vp)(\xi):=f_i(\vp (\xi)) \quad , \quad \vp \in \cb \ , \ \xi\in (0,1).$$
for smooth enough functions $f_i:\R \to \R$, $i=1,\ldots,d$.

\smallskip

Thus, Equation (\ref{eq-gene}) can here be written as
\begin{equation}\label{eq-nem}
y_t=S_t \vp+\int_0^t S_{t-u} (f_i(y_u)) \, dx^i_u \quad , \quad \vp\in \cb \ , \ t\in [0,T],
\end{equation}
or equivalently, in a multiparameter setting,
$$y(t,\xi)=\int_0^1 G_t(\xi,\eta) \vp(\eta) \, d\eta+\int_0^1 \int_0^t G_{t-u}(\xi,\eta) f_i(y(u,\eta)) \, dx^i_u d\eta \quad , \ t\in [0,T] , \xi \in (0,1),$$
where $G_t$ stands for the heat kernel on $(0,1)$ associated with Dirichlet boundary conditions.

\smallskip

\noindent
It is readily checked that if each $f_i$ belongs to $\cac^{1,\textbf{b}}(\R;\R)$ and $y\in \cacha^{0,\ka}(\cb_\ka)$ for some $\ka \in (\max (1-\ga,1/4),1/2)$, then the integral in the right-hand-side of (\ref{eq-nem}) can be interpreted with Proposition \ref{prop:yg-conv-int}. Indeed, owing to (\ref{compos}), we know that $f(y) \in \cac^0(\cb_\ka)$, while, due to the embedding (\ref{lien-der-delha}),
$$\cn[f(y);\cac^\ka(\cb)] \leq \norm{f'}_\infty \cn[y;\cac^\ka(\cb)] \leq c \norm{f'}_\infty \cn[y;\cacha^{0,\ka}(\cb_\ka)] < \infty.$$
For the remainder of the section, we shall rely on the following regularity assumptions.
\begin{hypothesis}\label{hyp:young-evol-eq}
We consider $ \ka \in (\max (1-\ga,1/4),1/2)$ and an initial condition $\vp\in \cb_\ka$. The family of functions $\{f_1,\ldots,f_d\}$ is such that $f_i$ is an element of $\cac^{3,\textbf{b}}(\R;\R)$ for $i=1,\ldots,d$.
\end{hypothesis}

In this context, the following existence and uniqueness result has been proven in \cite[Theorem 3.10]{RHE}:

\begin{proposition}\label{prop:exi-uni}
Under Hypothesis \ref{hyp:young-evol-eq}, Equation (\ref{eq-nem}) interpreted with Proposition \ref{prop:yg-conv-int} admits a unique solution $y\in \cacha^{0,\ka}(\cb_\ka)$, where we recall that the space  $\cacha^{0,\ka}(\cb_\ka)$ has been defined at Lemma \ref{lem:imbed-cacha}.
\end{proposition}

\smallskip

\noindent
As a preliminary step towards Malliavin differentiability of the solution to \eqref{eq-gene-fbm}, we shall study the dependence on $x$ of the deterministic equation (\ref{eq-nem}).

\subsection{Differentiability with respect to driving noise}
For equation (\ref{eq-nem}), consider the application
\begin{equation}\label{eq:def-Phi}
\Phi: \cac^\ga \to \cacha^{0,\ka}(\cb_\ka), 
\qquad x \mapsto y,
\end{equation}
for a given initial condition $\vp$. We shall elaborate on the strategy designed in \cite{nua-sau} in order to differentiate $\Phi$. Let us start with a lemma on linear equations:

\begin{lemma}\label{lem:exi-eq-li}
Suppose that $(x,y) \in \cac^\ga \times \cacha^{0,\ka}(\cb_\ka)$ and fix $t_0\in [0,T]$. Then for every $w \in \cacha^{0,\ka}([t_0,T];\cb_\ka)$, the equation
\begin{equation}\label{eq:lin-eq-nemytskii}
v_t=w_t+\int_{t_0}^t S_{t-u}(f_i'(y_u) \cdot v_u) \, dx^i_u \quad , \quad t\in [t_0,T],
\end{equation}
admits a unique solution $v\in \cacha^{0,\ka}([t_0,T];\cb_\ka)$, and one has
\begin{equation}\label{estim-eq-li}
\cn[v;\cacha^{0,\ka}([t_0,T];\cb_\ka)] \leq C_{x,y,T} \cdot \cn[w;\cacha^{0,\ka}([t_0,T];\cb_\ka)],
\end{equation}
where $C_{x,y,T}:=C(\norm{x}_\ga,\cn[y;\cacha^{0,\ka}(\cb_\ka)],T)$ for some function $C:(\R^+)^3 \to \R^+$ growing with its arguments.
\end{lemma}

\begin{proof}
The existence and uniqueness of the solution stem from the same fixed-point argument as in the proof of Proposition \ref{prop:exi-uni} (see \cite[Theorem 3.10]{RHE}), and we only focus on the proof of (\ref{estim-eq-li}).

\smallskip

\noindent
Let $I=[\ell_1,\ell_2]$ be a subinterval of $[t_0,T]$. One has, according to Proposition \ref{prop:yg-conv-int},
\begin{multline}\label{dem-lin}
\cn[v;\cacha^\ka(I;\cb_\ka)]\\
 \leq \cn[w;\cacha^\ka(\cb_\ka)]+c\lln I\rrn^{\ga-\ka} \norm{x}_\ga \lcl \cn[f_i'(y) \cdot v;\cac^0(I;\cb_\ka)]+\cn[f_i'(y) \cdot v;\cac^\ka(I;\cb)] \rcl.
\end{multline}
Now, by using successively (\ref{algebra}) and (\ref{compos}), we get
\bean
\cn[f_i'(y) \cdot v;\cac^0(I;\cb_\ka)] &\leq & c \cn[f_i'(y);\cac^0(\cb_\ka)] \cn[v;\cac^0(I;\cb_\ka)]\\
&\leq & c \lcl 1+\cn[y;\cac^0(\cb_\ka)] \rcl \cn[v;\cac^0(I;\cb_\ka)],
\eean
while, owing to (\ref{lien-der-delha}) and (\ref{sobol-incl}),
\bean
\lefteqn{\cn[f_i'(y) \cdot v;\cac^\ka(I;\cb)]}\\
&\leq & \cn[f_i'(y);\cac^\ka(I;\cb)] \cn[v;\cac^0(I;\cb_\infty)]+\cn[f_i'(y);\cac^0(I;\cb_\infty)] \cn[v;\cac^\ka(I;\cb)]\\
&\leq & c \lcl 1+\cn[y;\cacha^{0,\ka}(\cb_\ka)] \rcl \cn[v;\cacha^{0,\ka}(I;\cb_\ka)].
\eean
Going back to (\ref{dem-lin}), these estimates lead to
$$\cn[v;\cacha^\ka(I;\cb_\ka)] \leq \cn[w;\cacha^\ka(\cb_\ka)]+c_{x,y} \lln I\rrn^{\ga-\ka} \cn[v;\cacha^{0,\ka}(I;\cb_\ka)],$$
and hence
$$\cn[v;\cacha^{0,\ka}(I;\cb_\ka)] \leq \norm{v_{\ell_1}}_{\cb_\ka}+c \cn[w;\cacha^\ka(\cb_\ka)]+c_{x,y} \lln I\rrn^{\ga-\ka} \cn[v;\cacha^{0,\ka}(I;\cb_\ka)].$$
Control (\ref{estim-eq-li}) is now easily deduced with a standard patching argument.

\end{proof}

The following lemma on flow-type linear equations will also be technically important for our computations below.

\begin{lemma}\label{lem:flow-Psi}
Fix $(x,y) \in \cac^\ga \times \cacha^{0,\ka}(\cb_\ka)$ and for every $u \in [0,T]$, consider the system of equations
\begin{equation}
\Psi^i_{t,u}=S_{t-u}(f_i(y_u))+\int_u^t S_{t-w}(f_j'(y_w) \cdot \Psi^i_{w,u}) \, dx^j_w \quad , \quad t \in [u,T] \ , \ i \in \{1,\ldots,m\}.
\end{equation}
Then, for every $i\in \{1,\ldots,m\}$ and $t\in [0,T]$, the mapping $u\mapsto \Psi^i_{t,u}$ is continuous from $[0,t]$ to $\cb_\ka$. In particular, for every $\xi\in (0,1)$, $u\mapsto \Psi^i_{t,u}(\xi)$ is a continuous function on $[0,t]$. 
\end{lemma}

\begin{proof}
Let us fix $i\in \{1,\ldots,m\}$, $t\in [0,T]$. For any $0\leq u <v \leq t$, set
$$\Gamma^i_{v,u}(s):=\Psi^i_{s,v}-\Psi^i_{s,u} \quad , \quad s\in [v,T].$$
It is easy to check that $\Gamma^i_{v,u}$ is solution of the equation on $[v,T]$
$$\Gamma^i_{v,u}(s)=S_{s-v}(\Psi^i_{v,v}-\Psi^i_{v,u})+\int_v^s S_{s-w}(f_j'(y_w) \cdot \Gamma^i_{v,u}(w)) \, dx^j_w.$$
Therefore, according to the estimate (\ref{estim-eq-li}),
$$\norm{\Psi^i_{t,v}-\Psi^i_{t,u}}_{\cb_\ka}=\norm{\Gamma^i_{v,u}(t)}_{\cb_\ka} \leq \cn[\Gamma^i_{v,u};\cac^0(|v,T];\cb_\ka)]\leq c_{x,y,T} \norm{\Psi^i_{v,v}-\Psi^i_{v,u}}_{\cb_\ka}.$$ 
Now, observe that
$$\Psi^i_{v,v}-\Psi^i_{v,u}=f_i(y_v)-S_{v-u}(f_i(y_u))+\int_u^v S_{v-w}(f_j'(y_w) \cdot \Psi^i_{w,u}) \, dx^j_w,$$
and since $y\in \cacha^{0,\ka}(\cb_\ka)$, it becomes clear that $\norm{\Psi^i_{v,v}-\Psi^i_{v,u}}_{\cb_\ka} \stackrel{v\to u}{\longrightarrow} 0$.

\end{proof}

We now show how to differentiate a function which is closely related to equation \eqref{eq-nem}.

\begin{lemma}\label{lem:diff-F}
The application $F: \cac^\ga \times \cacha^{0,\ka}(\cb_\ka) \to \cacha^{0,\ka}(\cb_\ka)$ defined by
$$F(x,y)_t:=y_t-S_t \vp-\int_0^t S_{t-u}(f_i(y_u)) \, dx^i_u,$$
is differentiable in the Fréchet sense and denoting by $D_1F$ (resp. $D_2F$) the derivative of $F$ with respect to $x$ (resp. $y$), we obtain
\begin{equation}\label{der-part-1}
D_1F(x,y)(h)_t=-\int_0^t S_{t-u}(f_i(y_u)) \, dh^i_u,
\end{equation}
\begin{equation}\label{der-part-2}
D_2 F(x,y)(v)_t=v_t-\int_0^t S_{t-u}(f_i'(y_u) \cdot v_u) \, dx^i_u.
\end{equation}
Besides, for any $x\in \cac^\ga$, the mapping $D_2 F(x,\Phi(x))$ is a homeomorphism of $\cacha^{0,\ka}(\cb_\ka)$.

\end{lemma}

\begin{proof}
One has, for every $h\in \cac^\ga,v\in \cacha^{0,\ka}(\cb_\ka)$,
\begin{multline}\label{demo:diff}
F(x+h,y+v)_t-F(x,y)_t=\\
v_t-\int_0^t S_{t-u}(f_i'(y_u) \cdot v_u) \, dx^i_u-\int_0^t S_{t-u}(f_i(y_u)) \, dh^i_u-\lc R^1_t(v)+R^2_t(h,v)\rc,
\end{multline}
with
$$R^1_t(v):=\int_0^t S_{t-s} z^i_s \, dx^i_s \quad , \quad z^i_s:=\int_0^1 dr \int_0^1 dr' \, r \, f_i''(y_s+rr' v_s) \cdot v_s^2,$$
$$R^2_t(v,h):=\int_0^t S_{t-s} \tilde{z}^i_s \, dh^i_s \quad , \quad \tilde{z}^i_s:=\int_0^1 dr \, f_i'(y_s+rv_s) \cdot v_s,$$
and we now have to show that 
$$\cn[R^1_.(v)+R^2_.(h,v);\cacha^{0,\ka}(\cb_\ka)]=o\lp \lc \norm{h}_\ga^2+\cn[v;\cacha^{0,\ka}(\cb_\ka)]^2\rc^{1/2}\rp.$$

\smallskip

\noindent
Observe first that $\cn[R^1_.(v);\cacha^{0,\ka}(\cb_\ka)]\leq c \cn[R^1_.(v);\cacha^\ka(\cb_\ka)]$. Thanks to (\ref{algebra}) and (\ref{compos}), we get
$$\norm{z_s}_{\cb_\ka} \leq c \iint_{[0,1]^2} dr dr' \, \norm{f_i''(y_s+rr' v_s)}_{\cb_\ka} \norm{v_s}_{\cb_\ka}^2 \leq c \lcl 1+\norm{y_s}_{\cb_\ka}+\norm{v_s}_{\cb_\ka}\rcl \norm{v_s}_{\cb_\ka}^2.$$
Besides, owing to (\ref{sobol-incl}) and (\ref{lien-der-delha}) and setting $M_{ts}(r,r')=f_i''(y_t+rr' v_t)-f_i''(y_s+rr' v_s)$ we end up with
\begin{align*}
&\norm{z_t-z_s}_\cb \ \leq \ \iint_{[0,1]^2} dr dr' \,
\norm{M_{ts}(r,r')}_\cb \, \norm{v_s}^2_{\cb_\infty}+c \norm{v_t-v_s}_\cb \lcl \norm{v_t}_{\cb_\infty}+\norm{v_s}_{\cb_\infty} \rcl\\
&\leq  c \lcl \norm{y_t-y_s}_\cb+\norm{v_t-v_s}_\cb \rcl \norm{v_s}^2_{\cb_\ka}+c \norm{v_t-v_s}_\cb \lcl \norm{v_t}_{\cb_\ka}+\norm{v_s}_{\cb_\ka} \rcl.\\
&\leq  c \lln t-s \rrn^\ka \bigg\{ \lp \cn[y;\cacha^{0,\ka}(\cb_\ka)]+\cn[v;\cacha^{0,\ka}(\cb_\ka)] \rp \cn[v;\cac^0(\cb_\ka)]^2\\
&  \hspace{9cm}+\cn[v;\cacha^{0,\ka}(\cb_\ka)] \cn[v;\cac^0(\cb_\ka)] \bigg\}.
\end{align*}
The estimate (\ref{esti-yg}) for the Young convolutional integral now provides us with the expected control $\cn[R^1_.(v);\cacha^{0,\ka}(\cb_\ka)]=O( \cn[v;\cacha^{0,\ka}(\cb_\ka)]^2)$. In the same way, one can show that $\cn[R^2_.(h,v);\cacha^{0,\ka}(\cb_\ka)]=O( \norm{h}_\ga \cdot \cn[v;\cacha^{0,\ka}(\cb_\ka)])$, and the differentiability of $F$ is thus proved.

\smallskip

\noindent
Of course, the two expressions (\ref{der-part-1}) and (\ref{der-part-2}) for the partial derivatives are now easy to derive from (\ref{demo:diff}). As for the bijectivity of $D_2 F(x,\Phi(x))$, it is a consequence of Lemma~\ref{lem:exi-eq-li}.

\end{proof}

We are now ready to differentiate the application $\Phi$ defined by \eqref{eq:def-Phi}:

\begin{proposition}\label{prop:deriv-Phi}
The map $\Phi: \cac^\ga \to \cacha^{0,\ka}(\cb_\ka)$ is differentiable in the Fréchet sense. Moreover, for every $x\in \cac^\ga$ and $h\in \cac^\infty$, the following representation holds: if $t\in [0,T],\xi \in (0,1)$,
\begin{equation}\label{repres-diff}
D \Phi(x)(h)_t(\xi)=\int_0^t \Psi^i_{t,u}(\xi) \, dh^i_u,
\end{equation}
where $\Psi^i_{t,.}\in \cac([0,t];\cb_\ka)$ is defined through the equation
\begin{equation}
\Psi^i_{t,u}=S_{t-u}(f_i(\Phi(x)_u))+\int_u^t S_{t-w}(f_j'(\Phi(x)_w) \cdot \Psi^i_{w,u}) \, dx^j_w.
\end{equation}
\end{proposition}

\begin{proof}
Thanks to Lemma \ref{lem:diff-F}, the differentiability of $\Phi$ is a consequence of the implicit function theorem, which gives in addition
$$D \Phi (x)=-D_2 F(x,\Phi(x))^{-1} \circ D_1 F(x,\Phi(x)), \quad x\in \cac^\ga.$$
In particular, for every $x,h\in \cac^\ga$, $z:=D\Phi (x)(h)$ is the (unique) solution of the equation
\begin{equation}\label{eq-prop-diff}
z_t=\int_0^t S_{t-u}(f_i(\Phi(x)_u)) \, dh^i_u+\int_0^t S_{t-u}(f_i'(\Phi(x)_u) \cdot z_u) \, dx^i_u, \quad t\in [0,T].
\end{equation}
If $x\in \cac^\ga$ and $h\in \cac^\infty$, an application of the Fubini theorem shows (as in the proof of \cite[Proposition 4]{nua-sau}) that the path $\tilde{z}_t:=\int_0^t \Psi^i_{t,u} \, dh^i_u$ (which is well-defined thanks to Lemma \ref{lem:flow-Psi}) is also solution of (\ref{eq-prop-diff}), and this provides us with the identification (\ref{repres-diff}).

\end{proof}

As the reader might expect, one can obtain derivatives of any order for the solution when the coefficients of equation  \eqref{eq-nem} are smooth:
\begin{proposition}\label{prop:diff-ordre-supe}
Suppose that $f_i \in \cac^{\infty,\textbf{b}}(\R;\R)$ for every $i\in \{1,\ldots,m\}$. Then the function $\Phi: \cac^\ga \to \cacha^{0,\ka}(\cb_\ka)$ defined by \eqref{eq:def-Phi} is infinitely differentiable in the Fréchet sense. Moreover, for every $n\in \N^\ast$ and every $x,h_1,\ldots, h_n \in \cac^\ga$, the path $z_t:=D^n \Phi(x)(h_1,\ldots,h_n)$ satisfies a linear equation of the form
\begin{equation}\label{eq:diff-ordre-sup}
z_t=w_t+\int_0^t S_{t-u}(f_i'(\Phi(x)_u) \cdot z_u) \, dx^i_u \quad , \quad t\in [0,T],
\end{equation}
where $w\in \cacha^{0,\ka}(\cb_\ka)$ only depends on $x,h_1,\ldots,h_n$.
\end{proposition}

\begin{proof}
The details of this proof are omitted for the sake of conciseness, since they simply mimic the formulae contained in the proof of \cite[Proposition 5]{nua-sau}. As an example, let us just observe that for  $x,h,k \in \cac^\ga$, the path $z_t:=D^2 \Phi(x)(h,k)_t$ is the unique solution of~(\ref{eq:diff-ordre-sup}) with
\begin{multline*}
w_t:=\int_0^t S_{t-u}(f_i'(\Phi(x)_u) \cdot D\Phi(x)(h)_u) \, dk^i_u+\int_0^t S_{t-u}(f_i'(\Phi(x)_u) \cdot D\Phi(x)(k)_u) \, dh^i_u\\
+\int_0^t S_{t-u}(f_i''(\Phi(x)_u) \cdot D\Phi(x)(h)_u \cdot D\Phi(x)(k)_u) \, dx^i_u.
\end{multline*}

\end{proof}

\subsection{Existence of the density}\label{subsec:exi-density}

We will now apply the results of the previous section to an evolution equation driven by a fractional Brownian motion $B=(B^1,\ldots,B^d)$ with Hurst parameter $H>1/2$. Namely, we fix $\ka \in (\max(1/4,1-\ga),1/2)$ and an initial condition $\vp \in \cb_\ka$. We also assume that $f_i \in \cac^{3,\textbf{b}}(\R;\R)$ for $i=1,\ldots,m$.
We denote by $Y=\Phi(B)$ the solution of
\begin{equation}\label{eq:evol-fbm-nemytskii}
Y_t=S_t \vp+\int_0^t S_{t-u}(f_i(Y_u)) \, dB^i_u \quad , \quad t\in [0,T].
\end{equation}
Notice that since $H>1/2$ the paths of $B$ are almost surely $\ga$-H\"older continuous with H\"older exponent greater than $1/2$. Thus, equation \eqref{eq:evol-fbm-nemytskii} can be solved by a direct application of  Proposition \ref{prop:exi-uni}. Moreover, one can invoke Proposition~\ref{prop:deriv-Phi} in order to obtain the Malliavin differentiability of $Y_t(\xi)$:

\begin{lemma}
For every $t\in [0,T], \xi \in (0,1)$, $Y_t(\xi) \in \mathbb{D}^{1,2}_{\rm loc}$ and one has, for any $h\in \ch$,
\begin{equation}\label{ident-der-mal}
\lla \mathcal{D} (Y_t(\xi)), h \rra_{\ch}=D \Phi(B)(\mathcal{R}_H h)_t(\xi).
\end{equation}
\end{lemma}

\begin{proof}
According to \eqref{eq:rel-malliavin-pathwise-deriv}, we have that
$$\lla \mathcal{D} (Y_t(\xi)), h \rra_{\ch}=D (Y_t(\xi))(\mathcal{R}_H h)=\frac{d}{d\ep}_{|\ep=0} \Phi(B+\ep \mathcal{R}_H h)_t(\xi).$$
Furthermore, Proposition \ref{prop:deriv-Phi} asserts that $\Phi:\cac^\ga \to \cacha^{0,\ka}(\cb_\ka)$ is differentiable. Therefore
$$\frac{1}{\ep} \lc \Phi(x+\ep \mathcal{R}_H h)_t(\xi)-\Phi(x)_t(\xi) \rc=D \Phi (x)(\mathcal{R}_H h)_t(\xi)+\frac{1}{\ep} R(\ep \mathcal{R}_H h)_t(\xi),$$
with
$$\lln R(\ep  \mathcal{R}_H h)_t(\xi)\rrn  \leq \cn[R(\ep \mathcal{R}_H h);\cac^0(\cb_\infty)]\leq c \cn[R(\ep \mathcal{R}_H h);\cac^0(\cb_\ka)]=o(\ep),$$
and hence $\frac{d}{d\ep}_{|\ep=0} \Phi(B+\ep \mathcal{R}_H h)_t(\xi)=D \Phi (x)(\mathcal{R}_H h)_t(\xi)$, which trivially yields both the inclusion $Y_t(\xi) \in \mathbb{D}^{1,2}_{\rm loc}$ and expression \eqref{ident-der-mal}.

\end{proof}

With this differentiation result in hand plus some non degeneracy assumptions, we now obtain the existence of a density for the random variable $Y_t(\xi)$:
\begin{theorem}\label{theo-exi}
Suppose that for all $\la \in \R$, there exists $i\in \{1,\ldots,d\}$ such that $f_i(\la)\neq 0$. Then for all $t\in (0,1]$ and $\xi \in (0,1)$, the law of $Y_t(\xi)$ is absolutely continuous with respect to Lebesgue measure.
\end{theorem}

\begin{proof}
We apply Proposition \ref{prop:density-criterions} part (i), and we will thus prove that $\norm{\mathcal{D}(Y_t(\xi))}_{\ch}>0$ almost surely. Assume then that $\norm{\mathcal{D}(Y_t(\xi))}_{\ch}=0$. In this case, owing to (\ref{ident-der-mal}), we have $D \Phi (B)(\mathcal{R}_H h)_t(\xi)=0$ for every $h\in \ch$. In particular, due to (\ref{repres-diff}), one has $\int_0^t \Psi^i_{t,u}(\xi) \, dh^i_u=0$ for every $h\in \cac^\infty$. As $u\mapsto \Psi^i_{t,u}(\xi)$ is known to be continuous, it is easily deduced that $\Psi^i_{t,u}(\xi)=0$ for every $u\in [0,t]$ and every $i\in \{1,\ldots,d\}$, and so $0=\Psi^i_{t,t}(\xi)=f_i(Y_t(\xi))$ for every $i\in \{1,\ldots d\}$, which contradicts our non-vanishing hypothesis.

\end{proof}

\section{Smoothness of the density in the case of regularizing vector fields}\label{sec:smoothness-density}

Up to now, we have been able to differentiate the solution to \eqref{eq-nem} when the coefficients are fairly general Nemytskii operators. However, we have only obtained the inclusion $Y_t(\xi)\in\mathbb{D}^{1,2}_{\rm loc}$. Additional problems arise when one tries to prove $Y_t(\xi)\in\mathbb{D}^{1,2}$, due to bad behavior of linear equations driven by rough signals in terms of moment estimates. This is why we shall change our setting here, and consider an equation of the following type
\begin{equation}\label{equa-regu}
y_t =S_t \vp+\int_0^t S_{t-u} (L(f_i(y_u))) \, dx^i_u \quad , \quad t\in [0,T] \quad , \quad \vp \in \cb,
\end{equation}
where $x\in \cac^\ga([0,T];\R^d)$ with $\ga >1/2$, each $f_i:\R \to \R$ ($i\in \{1,\ldots, d\}$) is seen as a Nemytskii operator (see the beginning of Section \ref{sec:existence-density}), and $L$ stands for a regularizing linear operator of $\cb$. Let us be more specific about the assumptions in this section:

\begin{hypothesis}\label{hyp:regularized-young-eq}
We assume that for every $i\in \{1,\ldots, d\}$, $f_i$ is infinitely differentiable with bounded derivatives. Moreover, the operator $L:\cb \to \cb$ is taken of the form
$$L(\phi)(\xi):=\int_0^1 d\eta \, U(\xi,\eta) \phi(\eta),$$
for some positive kernel $U$ such that: $(i)$ $U$ is regularizing, i.e., $L$ is continuous from $\cb$ to $\cb_\la$ for every $\la \geq 0$, and $(ii)$ one has $c_U:=\min_{\xi \in (0,1)} \int_0^1 d\eta \, U(\xi,\eta) >0$. 
\end{hypothesis}

In other words, we are now concerned with the following equation on $[0,T] \times (0,1)$:
$$y(t,\xi)=\int_0^1 G_t(\xi,\eta) \vp(\eta) \, d\eta+\int_0^1\int_0^1 \int_0^t G_{t-u}(\xi,\eta)U(\eta,\mu) f_i(y(u,\mu)) \, dx^i_u d\mu d\eta, $$
with $U$ satisfying the above conditions $(i)$-$(ii)$.

\smallskip

This setting covers for instance the case of an (additional) heat kernel $U=G_\ep$ on $(0,1)$ for any fixed $\ep >0$. The following existence and uniqueness result then holds true:

\begin{proposition}\label{exi-cas-regu}
Under Hypothesis \ref{hyp:regularized-young-eq}, for any $\la \geq \ga$ and any initial condition $\vp\in \cb_\la$, Equation (\ref{equa-regu}) interpreted with Proposition \ref{prop:yg-conv-int} admits a unique solution in $\cacha^\ga(\cb_\la)$.
\end{proposition}

\begin{proof}
As in the proof of Proposition \ref{prop:exi-uni}, the result can be obtained with a fixed-point argument. Observe indeed that if $y\in \cacha^\ga(I;\cb_\la)$ ($I:=[\ell_1,\ell_2] \subset [0,1]$) and $z$ is the path defined by $z_{\ell_1}=y_{\ell_1}$, $z_t-S_{t-s}z_s=\int_s^t S_{t-u}(L(f_i(y_u))) \, dx^i_u$ ($s<t\in I$), then, according to Proposition \ref{prop:yg-conv-int}, $z\in \cacha^\ga(I;\cb_\la)$ and one has
\begin{equation}\label{dem:exi-cas-regu}
\cn[z;\cacha^\ga(I;\cb_\la)] \leq c \norm{x}_\ga \lcl \cn[L(f(y));\cac^0(I;\cb_\la^m)]+\lln I\rrn^\ga \cn[L(f(y));\cac^\ga(I;\cb_\la^m)] \rcl.
\end{equation}
Now, owing to the regularizing effect of $L$, it follows that $\cn[L(f(y));\cac^0(I;\cb_\la^m)] \leq \norm{L}_{\cl(\cb,\cb_\la)} \norm{f}_\infty$ and
$$\cn[L(f(y));\cac^\ga(I;\cb_\la^m)] \leq \norm{L}_{\cl(\cb,\cb_\la)} \norm{f'}_\infty \cn[y;\cac^\ga(I;\cb)] \leq c \cn[y;\cacha^{0,\ga}(I;\cb_\la)],$$
which, together with (\ref{dem:exi-cas-regu}), allows to settle the fixed-point argument.
\end{proof}

\

For the sake of clarity, we henceforth assume that $T=1$. The generalization to any (fixed) horizon $T>0$ easily follows from slight modifications of our estimates.

\smallskip

Moreover, for some technical reasons that will arise in the proofs of Propositions \ref{prop:contr-sys} and \ref{prop:eq-lin-reg}, we will focus on the case $\la=2+\ga$ in the statement of Proposition \ref{exi-cas-regu}. In other words, from now on, \emph{we fix the initial condition $\vp$ in the space $\cb_{2+\ga}$}.

\subsection{Estimates on the solution}
\label{sec:estim-solution}

Under our new setting, let us find an appropriate polynomial control on the solution to (\ref{equa-regu}) in terms of $x$.

\begin{proposition}\label{prop:contr-sys}
Suppose that $y$ is the solution of (\ref{equa-regu}) in $\cacha^\ga(\cb_{2+\ga})$ with initial condition $\vp$. Then there exists a constant $C_{\ga,f,L}$ such that
\begin{equation}\label{estim-poly}
\cn[y;\cacha^\ga([0,1];\cb_{2+\ga})] \leq C_{\ga,f,L} \lp 1+\norm{x}_\ga \rp\lp \max\lp \norm{x}_\ga^{1/\ga},\norm{\vp}_{\cb_{2+\ga}}^{1/2} \rp\rp^{1-\ga}.
\end{equation}
\end{proposition}

\begin{proof}
For any $N\in \N^\ast$, let us introduce the two sequences
$$\ep_k=\ep_{N,k}:=\frac{1}{N+k} \quad , \quad \ell_0:=0 \ , \ \ell_{k+1}=\ell^N_{k+1}:=\ell^N_k+\ep_{N,k}.$$
The first step of the proof consists in showing that we can pick $N$ such that for every $k$, 
\begin{equation}\label{cont-ep}
\ep^2_k \norm{y_{\ell_k}}_{\cb_{2+\ga}} \leq 1.
\end{equation}
For the latter control to hold at time $0$ (i.e., for $k=0$), we must first assume that $N\geq \norm{\vp}_{\cb_{2+\ga}}^{1/2}$. Now, observe that for any $k$, one has, owing to (\ref{esti-yg}),
\begin{eqnarray}
\lefteqn{\cn[y;\cacha^\ga([\ell_k,\ell_{k+1}];\cb_{2+\ga})]} \nonumber\\
&\leq & c \norm{x}_\ga \lcl \cn[L(f(y));\cac^0([\ell_k,\ell_{k+1}];\cb_{2+\ga}^m)]+\ep_k^\ga \cn[L(f(y));\cac^\ga([\ell_k,\ell_{k+1}];\cb_{2+\ga}^m)] \rcl \label{esti-regula-int}\\
&\leq & c \norm{x}_\ga \norm{L}_{\cl(\cb,\cb_{2+\ga})} \lcl 1+\ep_k^\ga \cn[y;\cac^\ga([\ell_k,\ell_{k+1}];\cb)] \rcl \nonumber\\
&\leq & c \norm{x}_\ga \lcl 1+\ep_k^\ga \cn[y;\cacha^\ga([\ell_k,\ell_{k+1}];\cb_{2+\ga})]+\ep_k^{\ga+2} \cn[y;\cac^0([\ell_k,\ell_{k+1}];\cb_{2+\ga})] \rcl\nonumber\\
&\leq & c_1 \norm{x}_\ga \lcl 1+\ep_k^\ga \cn[y;\cacha^\ga([\ell_k,\ell_{k+1}];\cb_{2+\ga})]+\ep_k^{\ga+2} \norm{y_{\ell_k}}_{\cb_{2+\ga}}\rcl,\nonumber
\end{eqnarray}
where we have used (\ref{lien-der-delha-2}) to get the third inequality. Consequently, if we take $N$ such that $2c_1 N^{-\ga} \norm{x}_\ga \leq 1$ (i.e. $N \geq (2c_1 \norm{x}_\ga)^{1/\ga}$), we retrieve
\begin{equation}\label{demo-regularise}
\cn[y;\cacha^\ga([\ell_k,\ell_{k+1}];\cb_{2+\ga})] \leq 2 c_1 \norm{x}_\ga+\ep_k^2 \norm{y_{\ell_k}}_{\cb_{2+\ga}}
\end{equation}
and hence
$$\norm{y_{\ell_{k+1}}}_{\cb_{2+\ga}} \leq 1+(1+\ep_k^{2+\ga}) \norm{y_{\ell_k}}_{\cb_{2+\ga}}.$$
From this estimate, if we assume that $\ep_k^2 \norm{y_{\ell_k}}_{\cb_{2+\ga}} \leq 1$, then 
$$\ep_{k+1}^2 \norm{y_{\ell_{k+1}}}_{\cb_{2+\ga}} \leq \ep_{k+1}^2+\ep_{k+1}^2 \lcl \norm{y_{\ell_k}}_{\cb_{2+\ga}}+\ep_k^\ga \rcl\leq 2\ep_{k+1}^2+\frac{\ep_{k+1}^2}{\ep_k^2}=\frac{2+(N+k)^2}{(N+k+1)^2} \leq 1$$
and (\ref{cont-ep}) is thus proved by induction. Going back to (\ref{demo-regularise}), we get, for every $k$,
\begin{equation}\label{unif-cont}
\cn[y;\cacha^\ga([\ell_k,\ell_{k+1}];\cb_{2+\ga})] \leq 2c_1 \norm{x}_\ga+1.
\end{equation}
By a standard patching argument, this estimate yields
$$\cn[y;\cacha^\ga([0,1];\cb_{2+\ga})] \leq \lcl 2c_1 \norm{x}_\ga+1 \rcl K^{1-\ga},$$
where $K$ stands for the smallest integer such that $\sum_{k=0}^K \ep_k\geq 1$. 

Finally, observe that $2\geq \sum_{k=0}^K \ep_k =\sum_{k=N}^{N+K} \frac{1}{k}$, and thus one can check that $K\leq (e^2-1)N \leq 7 N$. To achieve the proof of (\ref{estim-poly}), it now suffices to notice that $N$ can be picked proportional to $\max( \norm{x}_\ga^{1/\ga},\norm{\vp}_{\cb_{2+\ga}}^{1/2} )$.
\end{proof}

We now consider a linear equation, which is equivalent to \eqref{eq:lin-eq-nemytskii} in our regularized context:
\begin{equation}\label{eq-li-regul}
z_t=w_t+\int_0^t S_{t-u}(L(f_i'(y_u) \cdot z_u)) \, dx^i_u \quad , \quad t\in [0,1],
\end{equation}
where $w\in \cacha^{\ga}(\cb_{2+\ga})$ and $y$ stands for the solution of (\ref{equa-regu}) with initial condition $\vp \in \cb_{2+\ga}$.

\smallskip

The existence and uniqueness of a solution for (\ref{eq-li-regul}) can be proved along the same lines as Proposition \ref{exi-cas-regu}, that is to say via a fixed-point argument. We shall get a suitable exponential control for this solution.

\begin{proposition}\label{prop:eq-lin-reg}
There exists constants $C_1,C_2$ which only depend on $f$, $L$ and $\ga$ such that
\begin{equation}\label{estim-eq-lin}
\cn[z;\cac^0([0,1];\cb_{2+\ga})] \leq C_1 \cn[w;\cacha^{0,\ga}(\cb_{2+\ga})] \exp \lp C_2 \max \lp \norm{\vp}_{\cb_{2+\ga}}^{1/2},\norm{x}_\ga^{1/\ga} \rp \rp.
\end{equation}
Moreover, if $w_t=S_t \psi$ for some function $\psi \in \cb_{2+\ga}$, there exists an additional constant $C_3$ which only depend on $f$, $L$ and $\ga$ such that
\begin{multline}\label{estim-eq-lin-2}
\cn[z;\cacha^\ga([0,1];\cb_{2+\ga})]\\
 \leq C_3 \norm{\psi}_{\cb_{2+\ga}} \max \lp \norm{\vp}_{\cb_{2+\ga}}^{1/2},\norm{x}_\ga^{1/\ga} \rp \exp \lp C_2 \max \lp \norm{\vp}_{\cb_{2+\ga}}^{1/2},\norm{x}_\ga^{1/\ga} \rp \rp
\end{multline}
\end{proposition}

\begin{proof}
We go back to the notation $\ep_{N,k},\ell^N_k$ of the proof of Proposition \ref{prop:contr-sys}, and set, for every $c\geq 0$,
$$N(c):=\max \lp \norm{\vp}_{\cb_{2+\ga}}^{1/2},(2c\norm{x}_\ga)^{1/\ga} \rp.$$
We have seen in the proof of Proposition \ref{prop:contr-sys} that there exists a constant $c_1$ such that for every $N\geq N(c_1)$ and every $k$, one has simultaneously
\begin{equation}\label{rela-base}
\ep_{N,k}^2 \norm{y_{l^N_k}}_{\cb_{2+\ga}} \leq 1 \quad , \quad \cn[y;\cacha^\ga([\ell^N_{k},\ell^N_{k+1}] ;\cb_{2+\ga})] \leq 2 c_1 \norm{x}_\ga+1.
\end{equation}
Suppose that $N \geq N(c_1)$ and set $\cn_w:=\cn[w;\cacha^\ga(\cb_{2+\ga})]$. One has, similarly to (\ref{esti-regula-int}),
\bean
\lefteqn{\cn[z;\cacha^\ga([\ell_k^N,\ell_{k+1}^N];\cb_{2+\ga})]}\\
&\leq &\cn_w+ c \norm{x}_\ga \lcl \cn[f'(y) \cdot z;\cac^0([\ell_k^N,\ell_{k+1}^N];\cb^m)]+\ep_{N,k}^\ga \cn[f'(y) \cdot z;\cac^\ga([\ell_k^N,\ell_{k+1}^N];\cb^m)] \rcl\\
&\leq & \cn_w+c \norm{x}_\ga \big\{ \cn[z;\cac^0([\ell_k^N,\ell_{k+1}^N];\cb)]+\ep_{N,k}^\ga \cn[z;\cac^\ga([\ell_k^N,\ell_{k+1}^N];\cb)]\\
& & \hspace{4cm}+\ep_{N,k}^\ga \cn[y;\cac^\ga([\ell_k^N,\ell_{k+1}^N];\cb)] \cn[z;\cac^0([\ell_k^N,\ell_{k+1}^N];\cb_\infty)] \big\}\\
& \leq & \cn_w+c_2 \norm{x}_\ga \cn[z;\cac^0([\ell_k^N,\ell_{k+1}^N];\cb_{2+\ga})] \lcl 1+\ep_{N,k}^\ga \cn[y;\cac^\ga([\ell_k^N,\ell_{k+1}^N];\cb)] \rcl\\
& &\hspace{6cm} +c_2 \norm{x}_\ga \ep_{N,k}^\ga \cn[z;\cacha^\ga([\ell_k^N,\ell_{k+1}^N];\cb_{2+\ga})],
\eean
where we have used (\ref{sobol-incl}) and (\ref{lien-der-delha-2}) to derive the last inequality. Therefore, if we choose $N_2\geq \max\lp N(c_1), (2c_2 \norm{x}_\ga)^{1/\ga} \rp$, one has, for any $N\geq N_2$ and any $k$,
\begin{multline*}
\cn[z;\cacha^\ga([\ell_k^N,\ell_{k+1}^N];\cb_{2+\ga})]  \\
\leq 2 \cn_w+2c_2 \norm{x}_\ga \cn[z;\cac^0([\ell_k^N,\ell_{k+1}^N];\cb_{2+\ga})] \lcl 1+\ep_{N,k}^\ga \cn[y;\cac^\ga([\ell_k^N,\ell_{k+1}^N];\cb)] \rcl.
\end{multline*}
Thanks to (\ref{lien-der-delha-2}) and (\ref{rela-base}), we know that
\bean
\cn[y;\cac^\ga([\ell_k^N,\ell_{k+1}^N];\cb)] &\leq & c\lcl \cn[y;\cacha^\ga([\ell_k^N,\ell_{k+1}^N];\cb_{2+\ga})]+\ep_{N,k}^{2} \norm{y_{\ell_n^N}}_{\cb_{2+\ga}}\rcl\\
&\leq & c \lcl 2c_1 \norm{x}_\ga+2 \rcl. 
\eean
As a consequence, there exists $c_3$ such that for any $N\geq N_2$,
\begin{equation}\label{estim-loc}
\cn[z;\cacha^\ga([\ell_k^N,\ell_{k+1}^N];\cb_{2+\ga})] \leq 2\cn_w+c_3 \norm{x}_\ga \cn[z;\cac^0([\ell_k^N,\ell_{k+1}^N];\cb_{2+\ga})] \lcl 1+\ep_{N,k}^\ga \norm{x}_\ga \rcl.
\end{equation}
Then, for any $N\geq N_2$,
\begin{multline*}
\cn[z;\cac^0([\ell_k^N,\ell_{k+1}^N];\cb_{2+\ga})] \\
\leq 2\cn_w+\norm{z_{\ell_k^N}}_{\cb_{2+\ga}}+c_3 \norm{x}_\ga \ep_{N,k}^\ga \cn[z;\cac^0([\ell_k^N,\ell_{k+1}^N];\cb_{2+\ga})] \lcl 1+\ep_{N,k}^\ga \norm{x}_\ga \rcl.
\end{multline*}
Pick now an integer $N_3 \geq N_2$ such that
$$c_3 N_3^{-\ga} \norm{x}_\ga \lcl 1+N_3^{-\ga} \norm{x}_\ga \rcl \leq \frac{1}{2},$$
and we get, for any $k$,
$$\cn[z;\cac^0([\ell_k^{N_3},\ell_{k+1}^{N_3}];\cb_{2+\ga})] \leq 2\norm{z_{\ell_k^{N_3}}}_{\cb_{2+\ga}}+4\cn_w,$$
so $\cn[z;\cac^0([0,1];\cb_{2+\ga})] \leq  2^{K(N_3)} \norm{w_0}_{\cb_{2+\ga}}+2^{K(N_3)+2}\cn_w$, where $K(N_3)$ stands for the smallest integer such that $\sum_{k=0}^{K(N_3)} \ep_{N_3,k} \geq 1$. As in the proof of Proposition \ref{prop:contr-sys}, one can check that $K(N_3)\leq c N_3$. In order to get (\ref{estim-eq-lin}), it suffices to observe that there exists a constant $c_4$ such that any integer $N_3\geq c_4 \max ( \norm{\vp}_{\cb_{2+\ga}}^{1/2},\norm{x}_\ga^{1/\ga} )$ meets the above requirements.

\smallskip

Suppose now that $w_t=S_t \psi$. In particular, $\cn_w=0$. Then we go back to (\ref{estim-loc}) to obtain, thanks to (\ref{estim-eq-lin}),
$$\cn[z;\cacha^\ga([\ell_k^{N_3},\ell_{k+1}^{N_3}];\cb_{2+\ga})] \leq C_1 \norm{\psi}_{\cb_{2+\ga}} N_3^\ga  \exp \lp C_2 \max \lp \norm{\vp}_{\cb_{2+\ga}}^{1/2},\norm{x}_\ga^{1/\ga} \rp \rp,$$
which entails
\bean
\cn[z;\cacha^\ga([0,1];\cb_{2+\ga})] &\leq & C_1 \norm{\psi}_{\cb_{2+\ga}} N_3^\ga K(N_3)^{1-\ga} \exp \lp C_2 \max \lp \norm{\vp}_{\cb_{2+\ga}}^{1/2},\norm{x}_\ga^{1/\ga} \rp \rp\\
&\leq & C_3 \norm{\psi}_{\cb_{2+\ga}} N_3 \exp \lp C_2 \max \lp \norm{\vp}_{\cb_{2+\ga}}^{1/2},\norm{x}_\ga^{1/\ga} \rp \rp,
\eean
and (\ref{estim-eq-lin-2}) is thus proved.

\end{proof}

\begin{remark}
For any $t_0 \in [0,1]$, the proof of Proposition \ref{prop:eq-lin-reg} can be easily adapted to the equation starting at time $t_0$
$$z_t=w_{t,t_0}+\int_{t_0}^t S_{t-u}(L(f_i'(y_u) \cdot z_u)) \, dx^i_u \quad , \quad w_{.,t_0} \in \cacha^\ga([t_0,1]; \cb_{2+\ga}) \ , \ t \in [t_0,1],$$
and both estimates (\ref{estim-eq-lin}) and (\ref{estim-eq-lin-2}) remain of course true in this situation.
\end{remark}

\subsection{Smoothness of the density}
Let us now go back to the fractional Brownian situation
\begin{equation}\label{eq-regu-fbm}
Y_t=S_t \vp+\int_0^t S_{t-u}(L(f_i(Y_u))) \, dB^i_u \quad , \quad t\in [0,1] \ , \ \vp \in \cb_{2+\ga},
\end{equation}
where $\ga \in (\frac{1}{2},H)$ is a fixed parameter. We suppose, for the rest of the section, that the initial condition $\vp$ is fixed in $\cb_{2+\ga}$ and that Hypothesis \ref{hyp:regularized-young-eq} is satisfied. We denote by $Y$ the solution of (\ref{eq-regu-fbm}) in $\cacha^\ga(\cb_{2+\ga})$ given by Proposition \ref{exi-cas-regu}.

\smallskip

As in Subsection \ref{subsec:exi-density}, we wish to study the law of $Y_t(\xi)$ for $t\in [0,1]$ and $\xi \in (0,1)$. Without loss of generality, we focus more exactly on the law of $Y_1(\xi)$, for $\xi\in (0,1)$.

\smallskip

The first thing to notice here is that the whole reasoning of Section \ref{sec:existence-density} can be transposed without any difficulty to Equation (\ref{eq-regu-fbm}), which is more easy to handle due to the regularizing effect of $L$. Together with the estimates (\ref{estim-poly}) and (\ref{estim-eq-lin}), this observation leads us to the following statement:

\begin{proposition}
For every $\xi \in (0,1)$, $Y_1(\xi) \in \mathbb{D}^\infty$ and the law of $Y_1(\xi)$ is absolutely continuous with respect to the Lebesgue measure.
\end{proposition}

\begin{proof}
The absolute continuity of the law of $Y_1(\xi)$ can be obtained by following the lines of Section \ref{sec:existence-density}, which gives $Y_1(\xi) \in \D^\infty_{{\rm loc}}$ as well. Then, like in Proposition \ref{prop:diff-ordre-supe}, observe that $n$-th (Fréchet) derivatives $Z^n$ of the flow associated with (\ref{eq-regu-fbm}) satisfy a linear equation of the form
$$Z^n_t=W^n_t+\int_0^t S_{t-u}\lp L(f_i'(Y_u) \cdot Z^n_u) \rp \, dB^i_u, \quad t\in [0,1].$$
The explicit expression for $W^n$ ($n\geq 1$) can be derived from the formulae contained in \cite[Proposition 5]{nua-sau}, and it is easy to realize that due to (\ref{estim-poly}), one has $\cn[W^n;\cacha^{0,\ga}(\cb_{2+\ga})] \in L^p(\oom)$ for any $n$ and any $p$. Then, thanks to (\ref{estim-eq-lin}), we deduce that $\cn[Z^n;\cac^0(\cb_{2+\ga})]$ is a square-integrable random variable, which allows us to conclude that $Y_1(\xi) \in \D^\infty$ (see \cite[Lemma 4.1.2]{Nua}).
\end{proof}

\smallskip

The following proposition, which can be seen as an improvement of Lemma \ref{lem:flow-Psi} (in this regularized situation), provides us with the key-estimate to prove the smoothness of the density:

\begin{proposition}\label{prop:flow-psi}
For every $s\in [0,1]$, consider the system of equations
\begin{equation}\label{defi-flow-psi}
\Psi^i_{t,s}=S_{t-s}(L(f_i(Y_s)))+\int_s^t S_{t-u}(L(f_j'(Y_u) \cdot \Psi^i_{u,s})) \, dB^j_u \quad , \quad t\in [s,1] \ , \ i\in \{1,\ldots ,m\}.
\end{equation}
Then, for every $i\in \{1,\ldots,m\}$ and every $t\in [0,1]$, $\Psi^i_{t,.} \in \cac^\ga([0,t];\cb_{2+\ga})$. In particular, for any $\xi \in (0,1)$, $\Psi^i_{t,.}(\xi) \in \cac^\ga([0,t])$. 

\smallskip

Moreover, one has the following estimate
\begin{equation}\label{est-fond}
\cn[\Psi^i_{t,.};\cac^\ga([0,t];\cb_{2+\ga})] \leq Q(\norm{\vp}_{\cb_{2+\ga}},\norm{B}_\ga) \cdot \exp \lp c \max \lp \norm{\vp}_{\cb_{2+\ga}}^{1/2}, \norm{B}_\ga^{1/\ga} \rp \rp,
\end{equation}
for some polynomial expression $Q$.
\end{proposition}

\begin{proof}
As in the proof of Lemma \ref{lem:flow-Psi}, we introduce the path
$$\Gamma^i_{v,u}(s):=\Psi^i_{s,v}-\Psi^i_{s,u} \quad , \quad s\in [v,1] \ , \ 0\leq u <v \leq t,$$
and it is readily checked that $\Gamma^i_{v,u}$ solves the equation on $[v,1]$
$$\Gamma^i_{v,u}(s)=S_{s-v}(\Psi^i_{v,v}-\Psi^i_{v,u})+\int_v^s S_{s-w}(L(f_j'(Y_w) \cdot \Gamma^i_{v,u}(w))) \, dB^j_w.$$
Therefore, thanks to the estimate (\ref{estim-eq-lin}), we get
\begin{multline}\label{hold}
\norm{\Psi^i_{t,v}-\Psi^i_{t,u}}_{\cb_{2+\ga}}=\norm{\Gamma^i_{v,u}(t)}_{\cb_{2+\ga}}\leq  \cn[\Gamma^i_{v,u};\cac^0([v,1];\cb_{2+\ga})]\\
\leq c \norm{\Psi^i_{v,v}-\Psi^i_{v,u}}_{\cb_{2+\ga}} \exp \lp c \max \lp \norm{\vp}_{\cb_{2+\ga}}^{1/2}, \norm{B}_\ga^{1/\ga} \rp \rp.
\end{multline}
Then, by writing
$$\Psi^i_{v,v}-\Psi^i_{v,u}=L(f_i(Y_v)-f_i(Y_u))-\lc S_{v-u} -\id \rc(L(f_i(Y_u)))-\int_u^v S_{v-w}(L(f_j'(Y_w) \cdot \Psi^i_{w,u})) \, dB^j_w,$$
we deduce that
\begin{multline*}
\norm{\Psi^i_{v,v}-\Psi^i_{v,u}}_{\cb_{2+\ga}} \leq c \lln v-u \rrn^\ga \big\{ \norm{L}_{\cl(\cb,\cb_{2+\ga})} \cn[Y;\cac^\ga(\cb)]+\norm{L}_{\cl(\cb,\cb_{2+2\ga})}\\
+\norm{L}_{\cl(\cb,\cb_{2+\ga})}\norm{B}_\ga  \big( \cn[f'(Y) \cdot \Psi^i_{.,u};\cac^0(\cb^m)]+\cn[f'(Y) \cdot \Psi^i_{.,u};\cac^\ga(\cb^m)] \big)\big\}\\
\leq c \lln v-u\rrn^\ga \lcl 1+\norm{B}_\ga\rcl \lcl 1+\cn[Y;\cac^\ga(\cb)]\rcl \lcl 1+\cn[\Psi^i_{.,u};\cac^0(\cb_{2+\ga})]+\cn[\Psi^i_{.,u};\cac^\ga(\cb)] \rcl.
\end{multline*}
Going back to (\ref{hold}), the result now easily follows from the embedding $\cacha^{0,\ga}(\cb_{2+\ga}) \subset \cac^\ga(\cb)$ and the three controls (\ref{estim-poly}), (\ref{estim-eq-lin}) and (\ref{estim-eq-lin-2}).

\end{proof}

Proposition \ref{prop:flow-psi} implies in particular that the Young integral $\int_0^t \Psi^i_{t,u}(\xi) \, dh^i_u$ is well-defined for every $h\in \cac^\ga$, $t\in [0,1]$ and $\xi \in (0,1)$. We are thus in a position to apply the Fubini-type argument of \cite[Propositions 4 and 7]{nua-sau} so as to retrieve the following convenient expression for the Malliavin derivative:

\begin{corollary}
For every $\xi \in (0,1)$, the Malliavin derivative of $Y_1(\xi)$ is given by
\begin{equation}\label{expr-der-mal}
\cd_s^i(Y_1(\xi))=\Psi^i_{1,s}(\xi) \quad , \quad s\in [0,1] \ , \ i\in \{1,\ldots,m\},
\end{equation}
where $\Psi^i_{.,s}$ stands for the solution of (\ref{defi-flow-psi}) on $[s,1]$.
\end{corollary}

\begin{theorem}\label{thm:regu-density}
Suppose that there exists $\la_0 >0$ such that for every $i\in \{1,\ldots,m\}$ and every $\eta \in \R$, $f_i(\eta) \geq \la_0$. Then, for every $\xi \in (0,1)$, the density of $Y_1(\xi)$ with respect to the Lebesgue measure is infinitely differentiable.
\end{theorem}

\begin{proof}
We shall apply here the criterion stated at Proposition \ref{prop:density-criterions} item (ii). Notice that we already know that $Y_1(\xi) \in \mathbb{D}^\infty$, so it remains to show that for every $p\geq 2$, there exists $\ep_0(p) >0$ such that if $\ep < \ep_0(p)$, then $P\lp \norm{\cd_.(Y_1(\xi))}_{\ch} < \ep \rp \leq \ep^p$.

\smallskip

To this end, we resort to the following practical estimate, borrowed from \cite[Corollary 4.5]{bau-hai}: for every $\beta >H-1/2$, there exist $\al >0$ such that
\begin{equation}\label{estim-bau-hai}
\bp\lp \norm{\cd_.(Y_1(\xi))}_{\ch} < \ep \rp \leq \bp\lp \norm{\cd_.(Y_1(\xi))}_{\infty} < \ep^\al \rp
+\bp\lp \norm{\cd_.(Y_1(\xi))}_{\beta} > \ep^{-\al} \rp.
\end{equation}

The first term in the right-hand-side of (\ref{estim-bau-hai}) is easy to handle. Indeed, owing to the expression (\ref{expr-der-mal}) for the Malliavin derivative of $Y_1(\xi)$, one has
\bean
\norm{\cd_.(Y_1(\xi))}_\infty \ \geq \ \inf_{i=1,\ldots,m} | \Psi^i_{1,1}(\xi)|&=&\inf_{i=1,\ldots,m} | L(f_i(Y_1))(\xi) | \\
&=& \inf_{i=1,\ldots,m} \lln \int_0^1 d\eta \, U(\xi,\eta) f_i(Y_1(\eta)) \rrn \ \geq \ c_U \la_0 >0
\eean
(remember that $U$ and $c_U$ have been defined in Hypothesis \ref{hyp:regularized-young-eq}), so that $\bp( \norm{\cd_.(Y_1(\xi))}_{\infty} < \ep^\al )=0$ for $\ep$ small enough.

\smallskip

Then, in order to cope with $\bp\lp \norm{\cd_.(Y_1(\xi))}_{\beta} > \ep^{-\al} \rp$, one can simply rely on the Markov inequality, since, according to (\ref{est-fond}),
\bean
\norm{\cd_.(Y_1(\xi))}_\beta  \ = \ \norm{\Psi_{1,.}(\xi)}_\beta &\leq & c \sup_{i\in \{1,\ldots,m\}} \cn[\Psi^i_{1,.};\cac^\ga([0,1];\cb_{2+\ga})]\\
&\leq & c \, Q\lp \norm{\vp}_{\cb_{2+\ga}},\norm{B}_\ga \rp \cdot \exp\lp c \max \lp \norm{\vp}_{\cb_{2+\ga}},\norm{B}_\ga^{1/\ga} \rp \rp,
\eean
which proves that $\norm{\cd_.(Y_1(\xi))}_\beta \in L^q(\Omega)$ for every $q\geq 1$.

\end{proof}

\bibliography{mabiblio-rhe-regu-submitted}{}
\bibliographystyle{plain}

\end{document}